\documentclass[a4paper,11pt]{amsart}
\usepackage{amsfonts,amssymb}
\usepackage{graphicx,tikz}
\usepackage{tikz-cd} 
\usepackage{float}
\usepackage{systeme}
\usepackage{relsize}
\usepackage{amsmath, amsthm, amsfonts}
\usepackage{geometry}
\usepackage{verbatim}
\usepackage{hyperref}
\usepackage{enumitem}  
\usepackage[capitalise]{cleveref}
\usepackage{comment}
\usepackage{enumitem}
\usepackage{amsrefs}
\usepackage[official]{eurosym}
\usepackage{graphicx}
\usepackage{nomencl}
\usepackage{amsfonts}
\usepackage{hyperref}
\usepackage{amssymb}
\usepackage{fancyhdr}
\usepackage{amscd}
\usepackage[english]{babel}

\newtheorem{thm}{Theorem}[section]
\newtheorem{lem}[thm]{Lemma}
\newtheorem{prop}[thm]{Proposition}

\theoremstyle{definition}

\theoremstyle{definition}

\def\FF{\mathbb{F}}
\def\NN{\mathbb{N}}

\DeclareMathOperator{\ad}{ad}
\def\l{\langle}
\def\r{\rangle}

\keywords{Engel elements, Lie algebras}

\begin{document}

\title[Left 3-Engel elements in locally finite 2-groups]{Left $3$-Engel elements in locally finite $2$-groups}

\author[A. Hadjievangelou]{Anastasia Hadjievangelou}
\address{Department of Mathematical Sciences, University of Bath, Claverton Down, Bath BA2 7AY, United Kingdom}
\email{ah926@bath.ac.uk}

\author[G. Traustason]{Gunnar Traustason}
\address{Department of Mathematical Sciences, University of Bath, Claverton Down, Bath BA2 7AY, United Kingdom}
\email{gt223@bath.ac.uk}

\begin{abstract}
We give an infinite family of examples that generalise the construction given in \cite{GGM} of a locally finite 2-group $G$ containing a left $3$-Engel element $x$ where ${\langle x \rangle}^G$, the normal closure of $x$ in $G$, is not nilpotent. The construction is based on a family of Lie algebras that are of interest in their own right and make use of a classical theorem of Lucas, regarding when $\binom{m}{n}$ is even.
\end{abstract}

\maketitle

\section{Introduction}
Let $G$ be a group. An element $a\in G$ is a left Engel element in $G$, if for
each $x\in G$ there exists a non-negative integer $n(x)$ such that
      $$[[[x,\underbrace {a],a],\ldots ,a]}_{n(x)}=1.$$
If $n(x)$ is bounded above by $n$ then we say that $a$ is a left $n$-Engel element
in $G$. Throughout this paper we will assume that, when dealing with commutators or Lie products, these are left normed. Recall that the Hirsch-Plotkin radical of a group $G$ is the subgroup generated by all the normal locally nilpotent subgroups of $G$ and that this is also locally nilpotent.  It is straightforward to see that any element of the Hirsch-Plotkin
radical $HP(G)$ of $G$ is a left Engel element and the converse is known
to be true for some classes of groups, including solvable groups and 
finite groups (more generally groups satisfying the maximal condition on
subgroups) \cite{Gru, Baer}. The converse is however not true in general and this is the case
even for bounded left Engel elements. In fact whereas one sees readily that 
a left $2$-Engel element is always in the Hirsch-Plotkin radical this
is still an open question for left  $3$-Engel elements. Recently there has been a breakthrough and in \cite{Jab} it is shown that any left $3$-Engel element of odd order is contained in $HP(G)$. From \cite{Trac} one also knows that in order to generalise this to left $3$-Engel elements of any finite order it suffices to deal with elements of
order $2$. \\ \\
It was observed by William
Burnside \cite{Burn} that every element in a group of exponent $3$  is a left $2$-Engel
element and so the fact that every left $2$-Engel element lies in the Hirsch-Plotkin radical can be seen as the underlying reason why groups of exponent
$3$ are locally finite. For groups of $2$-power exponent there is a close link
with left Engel elements. If  $G$ is a group of exponent 
$2^{n}$ then it is not difficult to see that any element $a$ in $G$ of order $2$ is a left $(n+1)$-Engel element of $G$ (see the introduction of \cite{Trau1} for details). For sufficiently large $n$ we know that the variety of groups of exponent $2^{n}$ is not locally finite \cite{Ivan, Lys}. As a result one can see (for example in \cite{Trau1}) that it follows that for sufficiently large $n$ we do not have in general that a left $n$-Engel element is contained in the Hirsch-Plotkin radical. Using the fact that groups of exponent $4$ are locally finite \cite{San}, one can also see that if all left $4$-Engel elements of a group $G$ of exponent $8$ are in
$HP(G)$ then $G$ is locally finite. \\ \\
Swapping the role of $a$ and $x$ in the definition of a left Engel element we get the notion of a right Engel element. Thus an element $a\in G$ is a right Engel element, if for each $x\in G$ there exists a non-negative 
integer $n(x)$ such that 
    $$[[[a,\underbrace {x],x],\ldots ,x]}_{n(x)}=1.$$
If $n(x)$ is bounded above by $n$, we say that $a$ is a right $n$-Engel element. By a classical result
of Heineken \cite{Hein1} one knows that if $a$ is a right $n$-Engel element in $G$ then $a^{-1}$ is a left $(n+1)$-Engel
element. \\ \\
In \cite{New} M. Newell proved that if $a$ is a right $3$-Engel element in $G$ then $a\in HP(G)$ and in
fact he proved the stronger result that $\langle a\rangle^{G}$ is nilpotent of class at most $3$. The natural
question arises whether the analogous result holds for left $3$-Engel elements. In \cite{GGM} it is shown that  this is
not the case by giving an example of a locally finite $2$-group with a left $3$-Engel element $a$ such that
$\langle a\rangle^{G}$ is not nilpotent. Moreover in \cite{GAM} an example is given, for each odd prime $p$,  of a locally finite $p$-group containing a left 3-Engel element $x$ where $\langle x \rangle^G$ is not nilpotent. \\ \\
In  this paper we extend the example above, of a $2$-group, to an infinite family of examples. The construction will be based on
a family of Lie algebras that generalize the Lie algebra given in \cite{Trau2}. These algebras are of interest in their own right and will make use of a classical theorem of Lucas. Before stating Lucas's Theorem we need some notation. \\ \\
Let $p$ be a prime and consider non-negative integers $m$ and $n$ written in base $p$
\begin{eqnarray*}
m &=& m_0+m_1p+\dotsb +m_{k-1}p^{k-1}+ m_kp^k\\
n &=& n_0+n_1p+\dotsb +n_{k-1}p^{k-1}+ n_kp^k,
\end{eqnarray*}
where $0\leqslant m_0, \dotsc, m_k, n_0, \dotsc, n_k \leqslant p-1$.
We introduce a partial order $\leqslant_p$, where  $n\leqslant_p m$ if and only if $n_i \leqslant m_i$, for $0 \leqslant i \leqslant k$.

\begin{thm}[Lucas' Theorem] {The binomial coefficient $\binom{m}{n}$ is divisible by p  if and only if $n \nleqslant _p m$.}\label{Lucas}
\end{thm}

\noindent {\bf Remark.} Notice that when $p=2$ we get that the binomial coefficient $\binom{m}{n}$ is odd if and only if $n \leqslant _2 m$.

	\section{The Lie Algebra $L$}

In this section we construct a family of Lie algebras that extend the example given in  \cite{Trau2}. The construction makes an interesting use of Lucas' Theorem. \\

Let $\FF$ be the field of order 2 and let $n=2^m-2$, for $m \geqslant 2$. Consider the $(n+2)$-dimensional vector space $L=\FF v(0)+\FF v(1)+ \dotsb +\FF v(n-1) + \FF w + \FF x$. We equip $L$ with a binary product where

\begin{eqnarray*}
v(i) \cdot v(j) &=& \binom{j+1}{n-i} v(i \oplus j) \\
v(i) \cdot  w=w\cdot v(i) &=& v(i+1), \quad i=0, 1, \dotsc, n-2\\
v(n-1) \cdot w=w\cdot v(n-1) &=& w\\
v(i) \cdot x=x\cdot v(i) &=& 0 \\
w \cdot x=x\cdot w &=& v(0),
\end{eqnarray*}
such that $i \oplus j \equiv i+j \pmod{n-1}$ and $i \oplus j \in \{ 0, \dotsc, n-2 \}$, and where $ww=xx=0$. 
We then extend the product linearly on $L$.  The next theorem is our first main result. 

\begin{thm}
{L is a Lie algebra over $\FF$.}
\end{thm}

\begin{proof}
Let $0 \leqslant i, j, k \leqslant n-1$ and suppose that
\begin{eqnarray*}
i+1 &=&  a_0+2a_1 +\dotsb +2^{m-1}a_{m-1}\\
j+1 &=& b_0+2b_1+\dotsc +2^{m-1}b_{m-1}\\
k+1 &=& c_0+2c_1 +\dotsb +2^{m-1}c_{m-1}\\
n+1-(i+1) &=& (1-a_0)+2(1-a_1)+\dotsb +2^{m-1}(1-a_{m-1})\\
n+1-(j+1) &=& (1-b_0)+2(1-b_1)+\dotsb +2^{m-1}(1-b_{m-1})\\
n+1-(k+1) &=& (1-c_0)+2(1-c_1)+\dotsb +2^{m-1}(1-c_{m-1}),\\
\end{eqnarray*}
where $0 \leqslant a_i, b_i, c_i \leqslant 1$ for $0 \leqslant i \leqslant m-1$.
In order to show that the product is alternating, it only remains  to see that $v(i)\cdot v(i)=0$ and $v(i)\cdot v(j)=v(j)\cdot v(i)$ (recall that the characteristic is $2$). Firstly 
\[v(i) \cdot v(i)=\left(\! \begin{array}{c} i+1 \\ n+1-(i+1) \end{array} \! \right) v(i \oplus i).\]
In order for the product to be zero we know by Lucas' Theorem that we need $n+1-(i+1) \nleqslant_2 i+1$. Assume for contradiction
that $(n+1)-(i+1)\leqslant_2 i+1$. Then $1-a_i \leqslant a_i$, for all $0\leqslant i \leqslant m-1$ which implies that $a_0=a_1=\dotsb =a_{m-1}=1$. This gives $i+1=1+2+\dotsb +2^{m-1}=n+1$ that contradicts $i\leqslant n-1$. \\ \\
Now, for $v(j) \cdot v(i)=v(i)\cdot v(j)$, we need 
\[\left( \! \begin{array}{c} i+1\\ n-j \end{array} \! \right) = \left( \! \begin{array}{c} j+1\\ n-i \end{array} \! \right).\]
\[\begin{aligned}
\text{But, } \binom{i+1}{n+1-(j+1)} \text{ is odd }  &\iff 1-b_i \leqslant a_i, \quad \text{ for all } 0 \leqslant i \leqslant m-1\\
&\iff 1-a_i \leqslant b_i, \quad \text{ for all } 0 \leqslant i \leqslant m-1\\
&\iff \left( \! \begin{array}{c} j+1\\n+1-(i+1) \end{array} \! \right) \text{ is odd}.
\end{aligned}\]
Having established that the product is alternating we turn to the Jacobi identity .  If we have basis elements $e, f$ then, as $e\cdot e=0$ and $e\cdot f=f\cdot e$, we get  $(e\cdot e)\cdot f+(e\cdot f)\cdot e+(f\cdot e)\cdot e=2(e\cdot f)\cdot e=0$. Therefore, we only need to deal with the cases when the three basis elements are different. We will divide our anaysis into few cases. Notice first, that any Jacobi relation involving one occurrence of  $x$ and two $v(i)$'s, for $0 \leqslant i \leqslant n-1$, is clearly 0. There is one remaining type of a Jacobi relation that involves $x$. This is the one for a triple $(x,w,v(i))$, where $0\leq i\leq n-1$.  \\ \\
\noindent First, let $0 \leqslant i \leqslant n-2$. We have, $v(i) w x+ w x v(i) + x v(i) w = v(i +1) x + v(0) v(i) =\binom{i+1}{n}v(i)=0$, since $i+1  \leqslant n-1 < n$ and thus $n \nleqslant_2 i+1$. If $i=n-1$, then $v(n-1) w x+w x v(n-1)+ x v(n-1) w = w x + v(0) v(n-1) = v(0) +\binom{n}{n} v(0) = 0 $. \\ \\
Let us next consider triples of the type $(w,v(i),v(j))$ where $0 \leqslant i,j\leqslant n-1$. \\

\noindent {\bf Case 1:} Let $i=n-1$ and $0 \leqslant j\leqslant n-2$. Then, 
\begin{eqnarray*}
&&w v(n-1) v(j) + v(n-1) v(j) w +v(j) w v(n-1) \\
&=&  w v(j) +\binom{j+1}{1}v(j) w +v(j+1) v(n-1)\\
&=& v( j +1) + (j + 1) v(j+1) +\binom{j+2}{1} v(j +1)\\
&=& 2(j+2) v(j + 1)=0.
\end{eqnarray*}\\
{\bf Case 2:} Let $0 \leqslant i < j \leqslant n-2$. Then,
\begin{eqnarray*}
&&w v(i) v(j) + v(i) v(j) w+ v(j) w v(i) \\
&=& v(i +1) v(j) +\binom{j+1}{n-i} v(i\oplus j) w  + v(j+ 1) v(i) \\
&=& \binom{j+1}{n-i-1} v(i\oplus (j +1))+ \binom{j+1}{n-i}v(i \oplus (j + 1))+v(i) v(j+ 1)\\
&=& \binom{j+1}{n-i-1} v(i\oplus (j +1))+ \binom{j+1}{n-i} v(i \oplus (j + 1))+\binom{j+2}{n-i} v(i \oplus (j +1)) \\
&=& 2\binom{j+2}{n-i} v(i \oplus (j +1))=0,  
\end{eqnarray*}
where the last equality follows from Pascal's Rule. \\ \\

\noindent Finally, consider $v(i) v(j) v(k) + v(j) v(k) v(i) + v(k) v(i) v(j) = \alpha_1 v(i \oplus j \oplus k)+\alpha_2 v(i \oplus j \oplus k)+ \alpha_3 v(i \oplus j \oplus k)$, for $0 \leqslant i,j,k \leqslant n-1$.
Clearly, if all coefficients $\alpha_i$ are even then the Jacobi identity holds. So, assume without loss of generality that $\alpha_1$ is odd. Then, as $v(i)v(j) \neq 0$ we get $1\leqslant a_i+b_i$, for $i=0, \dotsc, m-1$. Then, $(i+1)+(j+1)=(a_0+b_0)+2(a_1+b_1)+\dotsb+2^{m-1} (a_{m-1}+b_{m-1})\geqslant 1+2+\dotsb+2^{m-1}=n+1$, so $i+j \geqslant n-1$.\\

\noindent {\bf Case 1:} Consider the case where $i+j=n-1$.
Then, $v(i) v(j) v(k) =v(0) v(k) =\binom{k+1}{n}v(k \oplus 0)$. Notice that $\binom{k+1}{n}$ is odd if and only if $k=n-1$. Hence $k=n-1$ and $v(i) v(j) v(k) = v(0 \oplus (n-1)) = v(0) $. Thus,
\begin{eqnarray*}
&&v(i) v(j) v(k) + v(j) v(k) v(i) + v(k) v(i) v(j) \\
&=& v(0) + (j+1) v(j) v(i) +(i+1) v(i)v(j) \\
&=& v(0) +(j+1) v(0) +(i+1) v(0) \\
&=& (n+2)v(0) = 2^{m} v(0) =0,
\end{eqnarray*}
as required.\\

\noindent{\bf Case 2:} Consider the case where $i+j \geqslant n$. We want to show that if $\alpha_1$ is odd then exactly one of $\alpha_2, \alpha_3$ is odd.  We shall consider each term separately. We have
\begin{eqnarray*}
&& v(i) v(j) v(k) =\binom{j+1}{n+1-(i+1)}v(i \oplus j) v(k) \\
&=&\binom{j+1}{n+1-(i+1)} \binom{k+1}{n+1-(i+j-n+2)} v(i \oplus j \oplus k), \text{ since } i+j \geqslant n\\
&=&\binom{j+1}{n+1-(i+1)} \binom{k+1}{2(n+1)-1-(i+1+j+1)} v(i \oplus j \oplus k).
\end{eqnarray*}
We assumed that $\alpha_1$ must be odd, hence we need both binomial coefficients to be odd. We have $\binom{j+1}{n+1-(i+1)}$ is odd if and only if $1 \leqslant a_i+b_i$, for all $i$. Let $t$ be the smallest index such that $a_t+b_t=1$. Therefore, $a_0=\dotsb=a_{t-1}=b_0=\dotsb=b_{t-1}=1$.\\
In order for $\binom{k+1}{2(n+1)-1-(i+1+j+1)}$ to be odd we must have that $2(n+1)-1-(i+1+j+1) \leqslant_2 k+1$. So,
\begin{eqnarray*}
&&2(n+1)-1-(i+1+j+1)\\
&=&2(1+2+2^2+\dotsb+2^{m-1})-1-((a_0+b_0)+2(a_1+b_1)+\dotsb+2^t+\\
&&\dotsb+2^{m-1}(a_{m-1}+b_{m-1}))\\
&=& 2^t-1+2^{t+1}(2-a_{t+1}-b_{t+1})+\dotsb+2^{m-1}(2-a_{m-1}-b_{m-1})\\
&=& 1+2+2^2+\dotsb+2^{t-1}+2^{t+1}(2-a_{t+1}-b_{t+1})+\dotsb\\
&& +2^{m-1}(2-a_{m-1}-b_{m-1})\\
&\leqslant_2& c_0+2c_1+\dotsb+2^{t-1}c_{t-1}+2^t c_t+2^{t+1}c_{t+1}+\dotsb+2^{m-1}c_{m-1},
\end{eqnarray*}
hence, $c_0=c_1=\dotsb=c_{t-1}=1$ and $2\leqslant a_i+b_i+c_i$, for all $i\geqslant t+1$. Notice that it follows in particular that 
$1\leq a_{i}+c_{i},b_{i}+c_{i}$ for all $i$ except possibly $i=t$. \\ \\
Notice that  the assumption $2(n+1)-1-(i+1+j+1) \leqslant_2 k+1$ implies that $2(n+1)-1-(i+1+j+1) \leqslant k+1$. That is $2(n-1)\leqslant i+j+k$. Then we must have that $j+k\geqslant n-1$ and $i+k\geqslant n-1$. If there is an equality, for example $j+k=n-1$, then, by symmetry, we have already shown in Case 1 that the Jacobi identity holds. So, we may assume that $j+k, i+k \geqslant n$.
Hence, similarly as above, $v(j) v(k) v(i) =\binom{k+1}{n+1-(j+1)} \binom{i+1}{2(n+1)-1-(k+1+j+1)} v(i \oplus j \oplus k)$ and $v(k) v(i) v(j) =\binom{i+1}{n+1-(k+1)}\binom{j+1}{2(n+1)-1-(k+1+i+1)} v(i \oplus j \oplus k)$.

\noindent We have two cases:
\begin{itemize}
\item[(a)]{\underline{$c_t=0$}:} Assume without loss of generality that $a_t=0$ and $b_t=1$.  Then, $a_t+c_t=0 \ngeqslant 1$ and $b_t+c_t=1$, hence $\binom{k+1}{n+1-(j+1)}$ is odd and $\binom{i+1}{n+1-(k+1)}$ is even. Notice that this implies that  the smallest index $s$ such that $b_s+c_s=1$ is $s=t$. Hence, $\binom{i+1}{2(n+1)-1-(j+1+k+1)}$ is odd, since $2\geqslant a_i +b_i+c_i$, for all $i \geqslant t+1$ and $a_0=\dotsb=a_{t-1}=1$. This shows that the Jacobi identity holds.

\item[(b)]{\underline{$c_t=1$}:} Assume without loss of generality that $a_t=0$ and $b_t=1$. Then, both $a_t+c_t, b_t+c_t \geqslant 1$, hence both binomials coefficients  $\binom{k+1}{n+1-(j+1)}$ and $\binom{i+1}{n+1-(k+1)}$ are odd. Now, similarly as in case (a), if $s$ is the smallest index such that $a_s+c_s=1$, then $s=t$ and so $\binom{j+1}{2(n+1)-1-(i+1+k+1)}$ is odd. It only remains to show that $\binom{i+1}{2(n+1)-1-(j+1+k+1)}$ is even. But, $b_t+c_t=2$, so the smallest index $l$ such that $b_l+c_l=1$ is $l\geqslant t+1$. Then, in order for $\binom{i+1}{2(n+1)-1-(j+1+k+1)}$ to be odd we require $a_0=a_1=\dotsb=a_t=1$, which contradicts our assumption, hence $\binom{i+1}{2(n+1)-1-(j+1+k+1)}$ must be even and the Jacobi identity holds.
\end{itemize}

\end{proof}

\begin{lem}
The Lie algebra $L$ has trivial center. 
\end{lem} 

\begin{proof}
    Take an element of $L$ say $l=\lambda_{-1} x+\lambda_0  v(0)+\dotsm + \lambda_{n-1} v(n-1)+ \mu w$, where $\lambda_{-1}, \lambda_0, \dotsc, \lambda_{n-1}, \mu \in \FF$, that lies in the center of $L$. Multiplying by $x$ gives $\mu v(0)=0$, therefore $\mu =0$. Then, multiplying by $w$ gives $\lambda_{-1}v(0)+\lambda_0v(1)+\dots+\lambda_{n-2}v(n-1)+\lambda_{n-1}w=0$ and therefore $\lambda_{-1}=\dots=\lambda_{n-1}=0$. 
\end{proof}

\begin{lem}
    $W=\FF v(0) +\dotsb+\FF v(n-1) +\FF w$ is a simple ideal of $L$. 
\end{lem}

\begin{proof}Consider the ideal $I$ generated by $y$, where $y=\lambda_0 v(0)+\dotsm + \lambda_{n-1}v(n-1) + \mu w$ and $\lambda_0, \dotsc, \lambda_{n-1}, \mu \in \FF$ are not all zero. We first show that $w\in I$. If $\lambda_{0}=\cdots = \lambda_{n-1}=0$, this  is clear. If not, take the smallest $i$ such that $\lambda_i \neq0$, where $0 \leqslant i \leqslant n-1$. Taking $y$ and mulitiplying $n-i$ times by $w$ gives us $\lambda_{i}w$ that  implies that $w \in I$. Having established that $w\in I$ we can multiply it  by $x,v(0), ...,v(n-2)$ to see that $v(0),\ldots ,
v(n-1)\in I$.  Hence $I = W$.
\end{proof}

Let $E=\langle \mbox{ad}(x), \mbox{ad}(v(0)), \mbox{ad}(v(1)), \dotsc, \mbox{ad}(v(n-1)), \mbox{ad}(w)\rangle \leqslant End(L)$. As $Z(L)$ is trivial, $E$ is the associative enveloping algebra of $L$.

\begin{lem}
The associative enveloping algebra $E$ is finite-dimensional.
\end{lem}

\begin{proof}
    This follows from the fact that $\mbox{dim}(L)=n+2$, hence $\mbox{dim (End}(L))=(n+2)^2$, thus $E$ must be of finite dimension.
\end{proof}
We will use $L$ to construct a locally nilpotent Lie algebra over $\FF$ of countably infinite dimension. This will then help us to construct a locally finite group $G$ with a left 3-Engel element $y$ where $\langle y \rangle ^G$ is not nilpotent. We now introduce a notation that was 
used in \cite{Trau2} of modified unions of subsets of $\NN$. We let
\[A  \sqcup B=\begin{cases} A \cup B, & \text{ if } A \cap B = \emptyset \\
\emptyset, & \text{ otherwise}
\end{cases}\]
For each non-empty subset $A$ of $\NN$ we let $W_A$ be a copy of the vector space $W=\FF v(0) +\dotsb+\FF v(n-1) +\FF w$, that is $W_A=\{ z_A : z \in W\}$ with addition $z_A+t_A=(z+t)_A$. We then take the direct sum of these 
\[W^* = \underset{ \emptyset \neq A \subseteq \NN}{\bigoplus} W_A\]
that we turn into a Lie algebra with multiplication
\[z_A \cdot t_B=(z\cdot t)_{A \sqcup B}\]
when $z_A\in W_A$ and $t_B \in W_B$ and extend linearly on $W^*$. The interpretation here is that $z_{\emptyset}=0$. Finally, we extend this to the semidirect product with $\FF x$
\[L^*=W^* \oplus \FF x\]
induced from the action $z_A \cdot x =(z\cdot x)_A$.
Notice that $L^*$ has basis 
\[\{x\} \cup \{v(0)_A, \dotsc, v(n-1)_A, w_A: \emptyset \neq A \subseteq \NN \}\]
and that 								
\[v(i)_A \cdot v(i)_B=w_A \cdot w_B=0,\]
\[v(i)_A \cdot x =0, \quad w_A \cdot x=v(0)_A,\]
for all $0 \leqslant i,j  \leqslant n-1$ and 
\[v(i)_A \cdot v(j)_B=\binom{j+1}{n-i} v(i \oplus j)_{A \sqcup B}, \text{ for all $0 \leqslant i, j \leqslant n-1$},\] \[v(i)_A \cdot w_B =v(i+1)_{A \sqcup B}, \text{ for all $0 \leqslant i, j \leqslant n-2$ }\] \[ \text{and } v(n-1)_A \cdot w_B= w_{A \sqcup B}.\]

\begin{lem} 
$L^{*}$ is locally nilpotent.
\end{lem}

\begin{proof} 
Notice that any finitely generated subalgebra of $L^*$ is contained in some 

\begin{eqnarray*}
    S & = & \langle x, v(0)_{A_1^0}, \dotsc, v(0)_{A_r^0},\dotsc, v(n-1)_{A_1^{n-1}}, \dotsc, v(n-1)_{A_t^{n-1}}, w_{B_1}, \dotsc,w_{B_l} \rangle.
\end{eqnarray*}
Thus it suffices to show that $S$ is nilpotent. Observe first that any Lie product with a repeated entry of $v(i)_{A}$ or $w_{B}$ is trivial and thus a non-trivial Lie product of the generators of $S$ can include in total at most $r+s+\cdots + t+l$ such elements.
As $v(i)_{A}\cdot x=(w_{B}\cdot x)\cdot x=0$ we have $(z\cdot x)\cdot x=0$ for all $z \in L^*$. Thus we see that $S$ is nilpotent of class at most $2(r+\dotsb +t+l)$. Therefore, $L^*$ is locally nilpotent.
\end{proof}

\section{The Group $G$}

For an element $y\in L^{*}$ we denote by $\mbox{ad}(y)$ the linear operator on $L^{*}$ induced by multiplication by $y$ on the right. 
In this section we find a group $G$ inside $\mbox{GL}(L^*)$ containing $1+\mbox{ad}(x)$, where $1+\mbox{ad}(x)$ is a left 3-Engel element in G, but where ${\langle 1+\mbox{ad}(x)\rangle}^G$ is not nilpotent. The next Lemma is a preparation for this.

\begin{lem}\label{lemma 2.2}
The adjoint linear operator $\mbox{ad}(x)$ on $L^*$ satisfies: \\ 
(a) $\mbox{ad}(x)^2=0$.\\
(b) $\mbox{ad}(x)\mbox{ad}(y)\mbox{ad}(x)=0$, for all $y \in L^*$.
\end{lem}

\begin{proof}
(a) This follows from our earlier observation that $(z\cdot x)\cdot x=0$ for all $z\in L^{*}$. \\
(b) Follows from $w_A \cdot x \cdot v(i)_B \cdot x= v(0)_A \cdot v(i)_B \cdot x=0,  w_A \cdot x \cdot w_B \cdot x= v(0)_A\cdot w_B \cdot x= v(1)_{A \sqcup B} \cdot x=0$.
\end{proof}

Now let $y$ be any of the generators $x, v(i)_A, w_A$, for $0 \leqslant i \leqslant n-1$. Since, $\mbox{ad}(y)^2=0$ for all $y$ it follows that 
\[(1+\mbox{ad}(y))^2=1+2\mbox{ad}(y)+\mbox{ad}(y)^2=1.\]
Thus, $1+\mbox{ad}(y)$ is an involution in $\mbox{GL}(L^*)$. Notice also that the following are elementary abelian $2$-groups of countably infinite rank. 

\begin{eqnarray*}
\mathcal{V}_0 &=& \langle1+\mbox{ad}(v(0)_A): A \subseteq \NN \rangle,\\ 
\mathcal{V}_1 &=& \langle1+\mbox{ad}(v(1)_A): A \subseteq \NN \rangle,\\ 
&\vdots& \\
\mathcal{V}_{n-1} &=& \langle 1+\mbox{ad}(v(n-1)_A): A \subseteq \NN \rangle,\\ 
\mathcal{W} &=& \langle1+\mbox{ad}(w_A): A \subseteq \NN \rangle
\end{eqnarray*}
We will be working with the group $G= \langle 1+\mbox{ad}(x), \mathcal{V}_0, \mathcal{V}_1, \dotsc, \mathcal{V}_{n-1}, \mathcal{W} \rangle$.

\begin{lem}\label{lemma 2.3}
The following commutator relations hold in $G$: \\
(a) $[1+\mbox{ad}(w_A),
1+\mbox{ad}(x)]=1+\mbox{ad}(v(0)_A)$. \\
(b) $[1+\mbox{ad}(v(i)_A), 1+\mbox{ad}(x)]=1$. \\
(c) $[1+\mbox{ad}(v(i)_A), 1+\mbox{ad}(v(j)_B)]=1+\binom{j+1}{n-i} \mbox{ad}(v(i \oplus j)_{A \sqcup B})$. \\
(d) $[1+\mbox{ad}(v(i)_A), 1+\mbox{ad}(w_B)]=1+\mbox{ad}(v(i+1)_{A \sqcup B})$, if $0 \leqslant i \leqslant n-2$. \\
(e) $[1+\mbox{ad}(v(n-1)_A), 1+\mbox{ad}(w_B)]=1+\mbox{ad}(w_{A \sqcup B})$. 
\end{lem}

\begin{proof}

(a) We have
\begin{eqnarray*}
&& [1+\mbox{ad}(w_A), 1+\mbox{ad}(x)] \\
&=& (1+\mbox{ad}(w_A))\cdot (1+\mbox{ad}(x)) \cdot (1+\mbox{ad}(w_A)) \cdot (1+\mbox{ad}(x))\\
&=& 1+\mbox{ad}(w_A)\mbox{ad}(x)+\mbox{ad}(x)\mbox{ad}(w_A)+\mbox{ad}(x)\mbox{ad}(w_A)\mbox{ad}(x)\\
&=& 1+\mbox{ad}(w_A\cdot x)\\
&=& 1+\mbox{ad}(v(0)_A),
\end{eqnarray*}
where we have used Lemma \ref{lemma 2.2}. Part (b) is proven similarly. For (c) we have
\begin{eqnarray*}
&&[1+\mbox{ad}(v(i)_A), 1+\mbox{ad}(v(j)_B)]\\
&=& (1+\mbox{ad}(v(i)_A)) \cdot (1+\mbox{ad}(v(j)_B) \cdot (1+\mbox{ad}(v(i)_A)) \cdot (1+\mbox{ad}(v(j)_B))\\
&=& 1+\mbox{ad}(v(i)_A)\mbox{ad}(v(j)_B)+\mbox{ad}(v(j)_B)\mbox{ad}(v(i)_A)\\
&=& 1+\mbox{ad}(v(i)_A \cdot v(j)_B)\\
&=& 1+\binom{j+1}{n-i} \mbox{ad}(v(i \oplus j)_{A \sqcup B}).
\end{eqnarray*}

Parts (d) and (e) are proved similarly.

\end{proof}

\noindent {\bf Remark.} Notice that as $L^*$ is locally nilpotent it follows from Lemma \ref{lemma 2.3} that $G$ is locally nilpotent. Next proposition clarifies the structure of $G$.

\begin{prop}
We have $G=\langle1+\mbox{ad}(x)\rangle \mathcal{V}_0 \dotsm \mathcal{V}_{n-1} \mathcal{W}$. In particular, for every element $g \in G$ there exists an expression $g=(1+\mbox{ad}(x))^{\epsilon} \cdot r_0 \dotsm r_{n-1} \cdot s$, with $\epsilon \in \{0,1\}, r_0 \in \mathcal{V}_0, \dotsc, r_{n-1} \in \mathcal{V}_{n-1}$ and $s \in \mathcal{W}$.
\end{prop}

\begin{proof}  
Suppose that
\[g=g_0(1+\mbox{ad}(x))g_1 \dotsm (1+\mbox{ad}(x))g_n\]
where $g_0, \dotsc, g_n$ are products of elements of the form $1+\mbox{ad}(v(0)_A), \dotsc, 1+\mbox{ad}(v({n-1})_A), 1+\mbox{ad}(w_A)$. From Lemma \ref{lemma 2.3} we know that $(1+\mbox{ad}(w_A))(1+\mbox{ad}(x))=(1+\mbox{ad}(x))(1+\mbox{ad}(w_A))(1+\mbox{ad}(v(0)_A))$ and $(1+\mbox{ad}(x))$ commutes with all products of the form $1+\mbox{ad}(v(i)_A)$, for $0 \leqslant i \leqslant n-1$. We can thus collect the $(1+\mbox{ad}(x))$'s to the left starting with the leftmost occurrence. This may introduce more elements of the form $1+\mbox{ad}(v(0)_A)$, but no new elements $1+\mbox{ad}(x)$. We thus have that 
\[g=(1+\mbox{ad}(x))^ng_1 \dotsm g_m\]
where $g_i$ is of the form $1+\mbox{ad}(v(j)_A)$, for $0 \leqslant j \leqslant n-1$ or of the form  $1+\mbox{ad}(w_A)$. Notice also that we can assume that $n=\epsilon$, where $\epsilon \in \{0,1\}$. This reduces our problem to the case when $g \in \langle \mathcal{V}_0, \dotsc, \mathcal{V}_{n-1}, \mathcal{W} \rangle$. 
Suppose
\[g=g_1g_2 \dotsm g_n,\]
where the terms $g_i$ are $(1+\mbox{ad}(v(0)_{A_1})), \dotsc, (1+\mbox{ad}(v(0)_{A_r})), \dotsc, (1+\mbox{ad}(v(n-1)_{B_1})), \dotsc, (1+\mbox{ad}(v(n-1)_{B_s})), (1+\mbox{ad}(w_{C_1})), \dotsc, (1+\mbox{ad}(w_{C_t}))$ in some order. 
By Lemma \ref{lemma 2.3} we have that $(1+\mbox{ad}(v(j)_B))(1+\mbox{ad}(v(i)_A))=(1+\mbox{ad}(v(i)_A))(1+\mbox{ad}(v(j)_B))(1+{\binom{j+1}{n-i}} \mbox{ad}(v(i \oplus j)_{A \sqcup B}))$ and $(1+\mbox{ad}(w_B))(1+\mbox{ad}(v(i)_A)=(1+\mbox{ad}(v(i)_A))(1+\mbox{ad}(w_B))(1+\mbox{ad}(v(i+1)_{A \sqcup B}))$, if $0 \leqslant i \leqslant n-2$ or $(1+\mbox{ad}(w_B))(1+\mbox{ad}(v(n-1)_A))=(1+\mbox{ad}(v(n-1)_A))(1+\mbox{ad}(w_B))(1+\mbox{ad}(w_{A\sqcup B}))$, if $i=n-1$. We can thus collect the terms so that

\begin{eqnarray*}
g &= & (1+\mbox{ad}(v(0)_{A_1})) \dotsm (1+\mbox{ad}(v(0)_{A_r})) \dotsm (1+\mbox{ad}(v(n-1)_{B_1})) \dotsm \\ && (1+\mbox{ad}(v(n-1)_{B_s})) \cdot
 (1+\mbox{ad}(w_{C_1})) \dotsm (1+\mbox{ad}(w_{C_t})) \cdot h_1 \dotsm h_m,
\end{eqnarray*}

\noindent where $h_i$ are of the form $1+\mbox{ad}(v(i)_D)$, with $0 \leqslant i \leqslant n-1$, or of the form $1+ad(w_D)$, where $D$ is a modified union of at least two sets from $$ \mathcal{S}=\{A_1, \dotsc, A_r, \dotsc, B_1, \dotsc, B_s, C_1, \dotsc, C_t\}.$$ Thus,
\[g=a_0 a_1 \dotsm a_{n-1} a h,\]
where $a_i \in \mathcal{V}_i$, for $0 \leqslant i \leqslant n-1$, $a \in \mathcal{W}$ and $h$ is a product of elements of the form $1+\mbox{ad}(v(i)_D)$, with $0 \leqslant i \leqslant n-1$, or of the form $1+ad(w_D)$, where $D$ is a modified union of at least two sets from $\mathcal{S}$. 

Repeating this collection process we get
\[g=b_0 b_1 \dotsm b_{n-1} b k,\]
where $b_i \in \mathcal{V}_i$, for $0 \leqslant i \leqslant n-1$, $b \in \mathcal{W}$ and $k$ is a product of elements of the form $1+\mbox{ad}(v(i)_E)$, where $0 \leqslant i \leqslant n-1$, or of the form $1+ad(w_E)$, where $E$ is a modified union of at least three sets from $\mathcal{S}$.

Continuing in this manner we conclude that after $k$ steps 
\[g=c_0 c_1 \dotsm c_{n-1} c f,\]
where $c_i \in \mathcal{V}_i$, for $0 \leqslant i \leqslant n-1$, $c \in \mathcal{W}$ and $f$ is a product of elements of the form $1+\mbox{ad}(v(i)_H)$, where $0 \leqslant i \leqslant n-1$, or of the form $1+ad(w_H)$, where $H$ is a modified union of at least $k+1$ sets from $\mathcal{S}$. However, any modified union of at least $r+\cdots +s+t+1 $ sets from $\mathcal{S}$ is trivial and thus $f=1$ when $k=r+\cdots +s+t$. This completes the proof.
\end{proof}

\begin{thm}\label{thm 2.6}
The element $1+\mbox{ad}(x)$ is a left 3-Engel element in $G$ such that ${\langle 1+\mbox{ad}(x) \rangle}^G$ is not nilpotent.
\end{thm} 

\begin{proof}
Showing that $1+\mbox{ad}(x)$ is a left $3$-Engel element in $G$ is equivalent to showing that $[(1+\mbox{ad}(x))^{g},1+\mbox{ad}(x)]$ commutes with $1+\mbox{ad}(x)$ for all $g\in G$. 
Let $g=h(1+\mbox{ad}(w_{A_1}))\dotsm (1+\mbox{ad}(w_{A_k}))$ be an arbitrary element in $G$, where $h \in \langle (1+\mbox{ad}(x)) \rangle \mathcal{V}_0 \dotsm \mathcal{V}_{n-1}$. We want to show that 
\[ [(1+\mbox{ad}(x))^g , 1+\mbox{ad}(x),  1+\mbox{ad}(x)]=1.\]
Let $y \in L^{*}$. Then
\[(1+\mbox{ad}(y))^{1+\mbox{ad}(w_A)}=(1+\mbox{ad}(w_A))(1+\mbox{ad}(y))(1+\mbox{ad}(w_A))=1+\mbox{ad}(y)+\mbox{ad}(y \cdot w_A).\]
Notice that $(1+\mbox{ad}(x))^g=(1+\mbox{ad}(x))^{(1+\mbox{ad}(w_{A_1}))\dotsm (1+\mbox{ad}(w_{A_n}))}$, since by Lemma \ref{lemma 2.3} we know that $1+\mbox{ad}(x)$ commutes with all elements of the form $1+\mbox{ad}(v(i)_B)$, for $0 \leqslant i \leqslant n-1$. Then, by induction we obtain that 
\[(1+\mbox{ad}(x))^g=1+\mbox{ad}(y)\]
where 
\begin{eqnarray*}
y &=& x + \displaystyle\sum_{1 \leqslant i \leqslant k} v(0)_{A_i} + \displaystyle\sum_{1 \leqslant i < j \leqslant k} v(1)_{A_i \sqcup A_j} + \dotsb + \displaystyle\sum_{1\leqslant i(1) < i(2) < \dotsc < i(n+1) \leqslant k} w_{A_{i(1)} \sqcup \dotsc \sqcup A_{i(n+1)}}.
\end{eqnarray*}
Since $\mbox{ad}(x) \mbox{ad}(y) \mbox{ad}(x)=0$, the commutator of $(1+\mbox{ad}(x))^g$ with $(1+\mbox{ad}(x))$ is 
\begin{eqnarray*} (1+\mbox{ad}(y))(1+\mbox{ad}(x))(1+\mbox{ad}(y))(1+\mbox{ad}(x))&=&1+\mbox{ad}(y)\mbox{ad}(x)+\mbox{ad}(x)\mbox{ad}(y) \nonumber \\ & & +\mbox{ad}(y)\mbox{ad}(x)\mbox{ad}(y). 
\end{eqnarray*}
Then, 
\begin{eqnarray*} 
[(1+\mbox{ad}(x))^g, 1+\mbox{ad}(x), 1+\mbox{ad}(x)] &=& [(1+\mbox{ad}(x))(1+\mbox{ad}(y))]^4\\
&=& (1+\mbox{ad}(y) \mbox{ad}(x)+ \mbox{ad}(x) \mbox{ad}(y)+ \mbox{ad}(y) \mbox{ad}(x) \mbox{ad}(y))^2\\
&=& 1,
\end{eqnarray*}
using the fact that $\mbox{ad}(x) \mbox{ad}(y) \mbox{ad}(x)=\mbox{ad}(y)^{2}=\mbox{ad}(x)^{2}=0$. \\ \\

However, the normal closure of $1+\mbox{ad}(x)$ in  $G$ is not nilpotent, as for $A_i=\{i\}$ we have

\begin{eqnarray*} 
&& [1+\mbox{ad}(w_{A_0}), 1+\mbox{ad}(x), 1+\mbox{ad}(w_{A_1}), 1+\mbox{ad}(w_{A_2}), \dotsc, 1+\mbox{ad}(w_{A_{n}}), \dotsc, \\
&& 1+\mbox{ad}(x), 1+\mbox{ad}(w_{A_{mn+1}}), 1+\mbox{ad}(w_{A_{mn+2}}),\dotsc, 1+\mbox{ad}(w_{A_{(m+1)n}})]\\
&=& 1+\mbox{ad}(w_B),
\end{eqnarray*}
where $B=A_0 \sqcup A_1 \sqcup \dotsc \sqcup A_{(m+1)n}=\{0, 1, 2, \dotsc, (m+1)n\}$.
\end{proof}

One might wonder if the nilpotency class of the subgroup generated by $r$ conjugates is unbounded for the  family we have  constructed. That turns out not to be the case. Our next aim is to show that the subgroup generated by any 
 $r$ conjugates $(1+\mbox{ad}(x))^{g_{1}},\ldots ,(1+\mbox{ad}(x))^{g_{r}}$ of $1+\mbox{ad}(x)$ in $G$ will be nilpotent of $r$-bounded class.  

We first work in a more general setting. For each $e\in E$ and $\emptyset \not =A\subseteq {\mathbb N}$, let $e(A)\in \mbox{End}(L^{*})$ where
\begin{eqnarray*} 
    v(i)_{B} \cdot e(A) &=& (v(i) \cdot e)_{B\sqcup A}\\
    w_B \cdot e(A) &=& (w \cdot e)_{B \sqcup A}\\
    x \cdot e(A) &=& (x \cdot e)_A,
\end{eqnarray*}
for $0 \leqslant i \leqslant n-1$. \\ \\
Then let $E^{*}=\langle \mbox{ad}(x), e(A):\,e\in E\mbox{ and }\emptyset \not =A\subseteq {\mathbb N}\rangle$. As $L^{*}$ is locally nilpotent, one sees readily that
the elements of $E^{*}$ are nilpotent and thus $1+E^{*}$ is a subgroup of $\mbox{End}(L^{*})$. \\


Despite the fact that the normal closure of $1+\mbox{ad}(x)$ in $G$ is not nilpotent, it turns out that the nilpotency class of the subgroup generated by any $r$ conjugates grows linearly with respect to $r$. In order to see this we must first introduce some more notation. In relation with the $r$ conjugates we let $A_{1}, A_{2}, \ldots ,A_{r}$ be any $r$ subsets of ${\mathbb N}$. 
For each $r$-tuple $(i1,i2,\ldots ,ir)$ of non-negative integers and each $e\in E$ we let 

 \begin{eqnarray*}
 e^{(i1,\ldots ,ir)} &=& \sum_{\tiny \begin{array}{l}B_{1}\subseteq A_{1}\\
 |B_1|=i1													\end{array}}
\ldots 
\sum_{\tiny \begin{array}{l}B_{k}\subseteq A_{r}\\
|B_r|=ir													\end{array}}e(B_{1}\sqcup B_{2}\sqcup \cdots \sqcup B_{r}).
\end{eqnarray*}
								                                
Notice that 
	\begin{eqnarray*}\label{1}
	    e^{(i1,\ldots ,ir)} \cdot f^{(j1,\ldots ,jr)} &=& {\tiny i1+j1\choose i1}\cdots {\tiny ir+jr\choose ir}\normalsize (ef)^{(i1+j1,\ldots ,ir+jr)}.
	\end{eqnarray*}

Consider the $r$ conjugates of $(1+\mbox{ad}(x))$ in $G$. Recall that each conjugate is of the from $(1+\mbox{ad}(x))^{(1+\mbox{ad}(w_{C_{1}})\cdots (1+\mbox{ad}(w_{C_{j}}))}$. Without loss of generality one can assume that each $C_{k}$ is a singleton set. The following argument also works for the more general case. Let 
     \[A_{1}=\{1,\ldots ,k_{1}\},\ A_{2}=\{k_{1}+1,\ldots ,k_{2}\},\ldots , A_{r}=\{k_{r-1}+1,\ldots ,k_{r}\}\]	
and
	\[e_1=\mbox{ad}(v(0)), e_2=\mbox{ad}(v(1)), \dotsc, e_n=\mbox{ad}(v(n-1)), e_{n+1}=\ad(w).\]															
Then we have seen (see the proof of Theorem \ref{thm 2.6}) that 
\begin{eqnarray*}
        (1+\mbox{ad}(x))^{(1+\mbox{ad}(w_{1}))\cdots (1+\mbox{ad}(w_{k_{1}}))} & = &
				   1+\mbox{ad}(x)+e_{1}^{(1,0,\ldots ,0)}+ \dotsm+ e_{n+1}^{(n+1,0,\ldots ,0)}\\
                                     & \vdots & \\
        (1+\mbox{ad}(x))^{(1+\mbox{ad}(w_{k_{r-1}+1}))\cdots (1+\mbox{ad}(w_{k_{r}}))}																		& = & 1+\mbox{ad}(x)+e_{1}^{(0,\ldots ,0,1)}+\dotsm + e_{n+1}^{(0,\ldots ,0,n+1)}.
\end{eqnarray*}

Let
\[f_{(i,k)}=e_i^{(0, \dotsc, 0,i,0,\dotsc, 0)},\]
where $i$ is the $k$-th coordinate and $1 \leqslant i \leqslant n+1$. Let 
\[F=\l f_{(j,k)}=e_j^{(0, \dotsc, 0, j, 0, \dotsc, 0)}: 1 \leqslant j \leqslant n+1, \quad 1\leqslant k \leqslant r \r .\]
Our aim is to find an upper bound for the nilpotence class of $F$. For this we need to understand better the two aspects of multplying
$ e^{(i1,\ldots ,ir)}$ and  $f^{(j1,\ldots ,jr)}$. These are

\begin{enumerate}
    \item[A.] Under which conditions the binomial coefficients are non-trivial ;
    \item[B.] Under which conditions is the Lie product $e_{i} \cdot e_{j}$ non trivial.
\end{enumerate}

\mbox{}\\
The next two lemmas will help clarify these questions. 
\begin{lem}\label{bczero}
    If $i+j \geqslant n+2$, then $\binom{i+j}{i}=0$,where $0 \leqslant i,j \leqslant n+1$.
\end{lem} 

\begin{proof}
Suppose that
\begin{align*}
    i&=\alpha_0 +2\alpha_1+\dotsc+2^{m-1}\alpha_{m-1}\\
    j&=\beta_0+2\beta_1+\dotsc+2^{m-1}\beta_{m-1}.
\end{align*}
We have that $\binom{i+j}{i}$ is odd if and only if $i \leqslant_2 i+j$. The latter happens if and only if there exists no $l$ such that $\alpha_l=\beta_l=1$, where $0 \leqslant l \leqslant m-1$.  In particular, for  the binomial coefficient to be non-zero we need $i+j \leqslant 1+2+\cdots 2^{m-1}=2^{m}-1=n+1$. 
\end{proof}

\begin{lem}
If $i+j \leqslant n-1$ and $1 \leqslant i,j \leqslant n$, then $e_i \cdot e_j=0$.
\end{lem}

\begin{proof} As a preparation we first show that $v(i)\cdot v(j)=0$ if $i+j \leq n-2$. To see this
let 
    \begin{align*}
        i+1&=a_0+2a_1+\dotsc+2^{m-1}a_{m-1}\\
        j+1&=b_0+2b_1+\dotsc+2^{m-1}b_{m-1}\\
        n+1-(i+1)&=(1-a_0)+2(1-a_1)+\dotsc+2^{m-1}(1-a_{m-1}),
    \end{align*}
    where $0 \leqslant a_l, b_l\leqslant 1$ for $0 \leqslant l  \leqslant m-1$. Then
    \begin{align*}
        \binom{j+1}{n+1-(i+1)} \mbox{ is odd} &\iff 1\leqslant a_l + b_l\mbox{ for all } 0\leq l\leq m-1.
\end{align*}
In particular for $\binom{j+1}{n+1-(i+1)}$ to be odd we need $ i+1+j+1 \geqslant 1+2+\cdots +2^{m-1}=2^{m}-1=n+1$. That is we need $i+j\geqslant n-1$. Thus if $i+j \leqslant n-2$ we have $v(i) \cdot v(j)=0$. Having established this preliminary result 
we turn to $e_{i} \cdot e_{j}$ where $1\leq i,j\leq n$. Firstly, 
\begin{eqnarray*}
    w e_ie_j &=& w v(i-1) v(j-1)= v(i) v(j-1) \\
    &=& \binom{j}{n+1-(i+1)} v(i \oplus (j-1))
\end{eqnarray*} 
and so it immediately follows from the result above that $w e_ie_j=0$, if $i+j-1\leqslant n-2$, that is if $i+j \leqslant n-1$. Then, for $0 \leqslant k \leqslant n-1$, we have
\begin{eqnarray*}
    v(k) e_i \cdot e_j &=&  v(k) v(i-1) v(j-1).
\end{eqnarray*}
From our preliminary  result above we know that $v(k)v(i-1)=0$, if $k+ (i-1) \leqslant n-2$. We can therefore assume that $k+(i-1) \geqslant n-1$ and hence $k \oplus (i-1) =k+(i-1)-(n-1)$. Therefore,
\begin{eqnarray*}
    v(k) e_i \cdot e_j &=& \binom{i}{n+1-(k+1)} v(k + (i-1)-(n-1)) v(j-1)\\
    &=& \binom{i}{n+1-(k+1)} v(k+i-n) v(j-1)\\
\end{eqnarray*}
which, by our preliminary result again, is trivial when $k+i-n+j-1 \leqslant n-2$, that is if $k+i+j \leqslant 2n-1$. Given that $0 \leqslant k \leqslant n-1$ we have in particular that the product is trivial when $i+j\leqslant n-1$.
 Hence, $e_i \cdot e_j=0$ if $i+j \leqslant n-1$.
\end{proof}

From these two lemmas we get the following result.
\begin{lem}
Let $1\leq i,j\leq n$. If  $f_{(i,k)} \cdot f_{(j,s)}$ is nonzero, then $n\leqslant i+j\leqslant n+1$. 
\end{lem}

We are now ready for establishing the linear upper bound for the nilpotence class of $F$.

\begin{lem}\label{lem4r}
  $F$ is nilpotent of a class at most $4r-1$.  
\end{lem}

\begin{proof}
Notice that by Lemma \ref{bczero} we have that $f_{(n+1,k)} \cdot f_{(j,s)}=0$ for any $1\leq j\leq n+1$. We thus only need to consider products of elements $f_{(i,k)}$ where $1\leq i\leq n$. Take any such product of even length $2u$ where $u$ is going to be determined later. Suppose the product is 
\begin{equation}\label{eq1}
    (f_{(i_1,k_1)} f_{(j_1,s_1)}) \dotsm (f_{(i_u,k_u)} f_{(j_u,s_u)}),
\end{equation}
where $1 \leqslant i_1, \dotsc,i_u, \dotsc, j_1, \dotsc, j_u \leqslant n$ and $1 \leqslant k_1, \dotsc, k_u, \dotsc, s_1, \dotsc, s_u \leqslant r$. We want to determine $u$ so that this product becomes 0. From Lemma 3.7 we know that  for this product
to be non-trivial we need  $ i_l+j_l \geqslant n$, for all $1 \leqslant l \leqslant u$. For $1\leq l\leq r$, let $t_l$ be the sum of all the $l$th coordinates of the superfixes of the $2u$ elements in the product.   For this to be non-zero we need
\begin{equation*}\label{eq2}
t_1+\dotsc +t_r= (i_1+j_1)+ \dotsm + (i_u + j_u) \geqslant nu. 
\end{equation*}
If one of $t_1,\dotsc, t_r$ is greater than $n+1$ then Lemma 3.5 implies that the product is zero. We need to find how big $u$ has to be so that this happens. Notice that the largest value of $t_1, \dotsc, t_r$ is greater or equal to the mean value and this is at least $\frac{nu}{r}$. Therefore, it suffices that 
\[ \frac{nu}{r} \geqslant n+2.\]
This holds when $u=2r$. Hence, $F^{2(2r)}=F^{4r}=0$.
\end{proof}

From this it is not difficult to obtain a linear upper bound for nilpotence class of the group generated by $r$ conjugates of
$1+\mbox{ad}(x)$. First we extend the analysis of $F$ to the subalgebra 
$Q=\l \ad(x), F \r $ of $E^*$.  If no element $f_{(n+1,k)}$ occurs then any product of elements of $Q$ of length $4r+1$ is trivial. Here we are using the fact that $\mbox{ad}(x)^{2}=0$, Lemma \ref{lem4r} and the  fact that $\ad(x)$ commutes with elements of the form $f_{(i,l)}$ when $1\leq i\leq n$. Suppose therefore that at least one element $f_{(n+1,k)}$ is involved in a product of elements of $Q$. If an element of the form $f_{(i,l)}$ included, where $1 \leqslant i \leqslant n$, we can pick $f_{(n+1,k)}$ and $f_{(i,l)}$ that are of closest distance within the product. This reduces things to the following situations:
\begin{eqnarray*}
    && \dotsm f_{(n+1,k)}  \ad(x) f_{(i,l)} \dotsm,\\
    && \dotsm f_{(i,l)}  \ad(x) f_{(n+1,k)} \dotsm,\\
    && \dotsm f_{(n+1,k)} f_{(i,l)} \dotsm,\\
    && \dotsm f_{(i,l)} f_{(n+1,k)} \dotsm.
\end{eqnarray*}
These  products are trivial, as $\ad(x)$ commutes with $f_{(i,l)}$ and the products of $f_{(n+1,k)}$ and  $f_{(i,l)}$ are trivial  by Lemma \ref{bczero}. We can therefore assume that no element $f_{(i,l)}$ is involved, where $1\leq i\leq n$.  Then we only have a product in $\ad(x)$'s and elements of the form $f_{(n+1,k)}$. By Lemma 3.5 and the fact that $\ad(x)^2=0$ the product needs to alternate  between $\mbox{ad}(x)$ and elements of the form $f_{(n+1,k)}$. We will make use of the fact that 
\[\ad(x) f_{(n+1,k)}= f_{(n+1,k)}\ad(x) + f_{(1,k)}. \]
We have
\begin{eqnarray*}
\ad(x) f_{(n+1,k_1)}  \ad(x) f_{(n+1,k_2)} &=& (f_{(n+1,k_1)}\ad(x) + f_{(1,k_1)}) \ad(x) f_{(n+1, k_2)}\\
&=& f_{(n+1,k_1)}\ad(x)^2f_{(n+1, k_2)} + f_{(1,k_1)} \ad(x) f_{(n+1, k_2)}\\
&=& 0+ \ad(x)f_{(1,k_1)} f_{(n+1, k_2)}=0,
\end{eqnarray*}
where the last equality follows from Lemma \ref{bczero}. Similarly any product of the form $f_{(n+1,k_1)}  \ad(x) f_{(n+1,k_2)} \ad(x)$ is $0$. To conclude we have seen that $Q^{4r+1}=0$. Now let $H$ be a subgroup of $G$ generated by any $r$ conjugates $(1+\mbox{ad}(x))^{g_{1}},\ldots ,(1+\mbox{ad}(x))^{g_{r}}$ of $1+\mbox{ad}(x)$ in $G$. Then
\[\gamma_{4r+1}(H) \leqslant 1+Q^{4r+1}=1.\]
Thus we have the following result.

\begin{prop}Let $(1+\mbox{ad}(x))^{g_{1}},\ldots ,(1+\mbox{ad}(x))^{g_{r}}$ be any $r$ conjugates of $1+\mbox{ad}(x)$ in $G$. Then $H=\l (1+\mbox{ad}(x))^{g_{1}},\ldots ,(1+\mbox{ad}(x))^{g_{r}} \r$ is nilpotent of class at most $4r+1$. 
\end{prop}

{\it Acknowledgement}. The first author is partially supported by `The Norton Scholarship'. We acknowledge the EPSRC (grant number 16523160) for support. Moreover, we would like to thank Marialaura Noce for drawing our attention to Lucas' Theorem.

\end{document}